\documentclass[11pt]{article}

\usepackage{amsmath}
\usepackage[english]{babel}
\usepackage{amsbsy}
\usepackage{amsfonts}
\usepackage{enumerate}
\usepackage{amssymb}
\usepackage{amsthm}
\usepackage[pdftex]{graphicx}
\usepackage{float}
\usepackage{framed}
\usepackage{url}
\usepackage{mathtools}
\usepackage{mathrsfs}
\usepackage{geometry}
\usepackage[T1]{fontenc}
\usepackage[utf8]{inputenc}
\usepackage{authblk}
\usepackage[numbers]{natbib}

\numberwithin{equation}{section}
\newcommand{\mbb}[1]{\mathbb{#1}}

\newcommand{\la}{\lambda}
\newcommand {\sem}[1]{\mbox{$\{{{#1}(t)}\}_{t \geq 0}$}}
\newtheorem{theorem}{\textbf{Theorem}}[section]
\newtheorem{corollary}[theorem]{\textbf{Corollary}}
\newtheorem{lemma}[theorem]{Lemma}
\newtheorem{remark}[theorem]{\textbf{Remark}}

\newtheorem{proposition}{\textbf{Proposition}}[section]

\newtheorem{example}{\textbf{Example}}[section]

\title{Analysis and Simulations of the Discrete Fragmentation Equation with Decay\footnote{The paper was presented at the BIOMATH 2017 Conference, Skukuza, 25--30.06.2017. The research and the conference attendance was supported from the funds of the DST/NRF SARChI Chair in Mathematical Models and Methods in Biosciences and Bioengineering and NRF PhD bursary (LOJ). It is an updated version of the paper \textit{Analysis and Simulations of the Discrete Fragmentation Equation with Decay}, M2AS, doi.org/10.1002/mma.4666  }}
\author[1]{J. Banasiak}
\author[2]{L.O. Joel}
\author[3]{S. Shindin}
 \affil[1]{Department of Mathematics and Applied Mathematics\\
 	University of Pretoria\\
 	Pretoria, South Africa \& Institute of Mathematics, Technical University of \L\'{o}d\'{z}, \L\'{o}d\'{z}, Poland\\e-mail:jacek.banasiak@up.ac.za}\date{}
\affil[2,3]{School of Mathematics, Statistics and Computer Science\\
           University of Kwazulu-Natal\\
           Westville Campus, Durban\\
           South Africa}\date{}

\newcommand{\mb}[1]{\boldsymbol{#1}}

\begin{document}

\maketitle

\pagestyle{myheadings}
\markboth{J. Banasiak, L.O. Joel and S. Shindin}{Discrete Decay-Fragmentation Equation}
\thispagestyle{empty}

\begin{abstract}
\noindent
Fragmentation--coagulation processes, in which aggregates can break up or get together,  often occur together with decay processes in which the components can be removed from the aggregates by a chemical reaction, evaporation, dissolution, or death. In this paper we consider the discrete decay--fragmentation equation and prove the existence and uniqueness of physically meaningful solutions to this equation using the theory of semigroups of operators. In particular, we find conditions under which the solution semigroup is analytic, compact and has the asynchronous exponential growth property.  The theoretical analysis is illustrated by a number of numerical simulations.
\end{abstract}


\section{Introduction}

Fragmentation--coagulation processes, in which we observe breaking up of  clusters of particles into smaller pieces or, conversely, creation of bigger clusters by an aggregation of smaller pieces, occur in many areas of  science and engineering, where they describe polymerization and depolymerization, droplets formation and their breakup, grinding of rocks, formation of animal groups, or phytoplankton aggregates, \cite{Ziff1985, Ziff1980, Ziff1991, Banasiak2006, jackson1990model, Oku}. In many cases, fragmentation and coagulation are accompanied by other processes such as growth or decay of clusters due to chemical reactions, surface deposition from the solute or, conversely, dissolution and evaporation, or birth and death of cells forming the cluster, see e.g. \cite{Ackleh1997, Cai1991, Huang1991, Banasiak2006, Banasiak2003a}. Another process affecting the concentration of clusters is their sinking, or sedimentation, \cite{Ackleh1997, jackson1990model}.

There are two main ways  of modelling fragmentation--coagulation processes: the discrete one, in which we assume that each cluster is composed of a finite number of identical indivisible units called monomers, \cite{Smoluchowski1916}, and the continuous one, where it is assumed that the size of the particles constituting the cluster can be an arbitrary positive number $x\in \mathbb R$, \cite{Muller1928}. Consequently, the latter case is modelled by an integro--differential equation for the density of size $x$ clusters, while in the former we deal with an infinite system of ordinary differential equations for the densities of the clusters of size $i \in \mathbb N$, also called $i$-mers. Similarly, the growth/decay process is modelled by a first order (transport) differential operator in $x$, \cite{Cai1991, Huang1991, Banasiak2006}, in the continuous case and by a difference operator, as in the birth-and-death equation, in the discrete case.

We assume that the mass of the monomer is normalized to 1 and thus the term \textit{size} is used interchangeably with \textit{mass}.

In this paper, we shall focus on the discrete fragmentation model with death and sedimentation, given by the system
 \begin{equation}
\begin{aligned}
\label{gf1b}
&\frac{d f_{i}}{d t} = r_{i+1}f_{i+1} - r_{i}f_{i} - d_{i}f_{i} -  a_{i}f_{i} + \sum_{j= i+1}^{\infty} a_{j} b_{i,j} f_{j}, \quad i \geq 1, \\
&f_{i}(0) = \mathring f_{i}, \quad i \geq 1,
\end{aligned}
\end{equation}
where $\mb f = (f_i)_{i=1}^\infty$ gives the numbers $f_i$ of $i$-mers, $i \geq 1$, $ r_{i} $ and $ d_{i} $ represent, respectively, the decay and sedimentation coefficients, $ r_{i} > 0 $ and $ d_{i} \geq 0 $. The fragmentation rate is given by $ a_{i}, $ while $ b_{i,j} $ is the average number of $i$-mers produced after the breakup of a $j$-mer, with $ j \ge i$. The difference operator, $\mb f \rightarrow (r_{i+1}f_{i+1} - r_{i}f_{i})_{i=1}^\infty$, gives the rate of change of the number of $i$-mers due to the decay/death process (for instance, assuming that in an aggregate of cells in a short period of time only one monomer may die, the number of $i$-mers increases due to the death of cells in the  $i+1$-mers, which then become $i$-mers,  and decreases due to the death of cells in size $i$-mers that then move to the $i-1$ class). Setting $r_{i} = 0$ and $d_{i} = 0 $, we arrive at the classical mass-conserving fragmentation equation.

Naturally, the clusters can only fragment into smaller pieces. Hence, we must have
\begin{equation}
\label{eq11}
 a_{1} = 0, \qquad  b_{i, j} = 0,  \quad  i \geq  j.
\end{equation}
We also assume that all clusters that are not monomers undergo fragmentation; that is,  $a_{i} > 0$ for $i \ge 2$. Since the fragmentation process only consists in the rearrangement of the total mass into clusters, it must be conservative and for this we require
\begin{equation}
\label{equ12}
\sum_{i=1}^{j-1} i b_{i,j} = j, \quad j \geq 2.
\end{equation}
The above equation expresses the fact that the masses of all particles resulting from a break-up of a cluster of mass $j$ must add up to $j$.  Note that, the total mass of the ensemble at time $t$ is given by
$$M(t) = \sum_{i=1}^{\infty} if_{i}
$$
and, in general, for the system \eqref{gf1b} is not conserved.

\begin{remark} We observe that it is possible to include the terms $r_{i+1}u_{i+1}$ into the gain term of \eqref{gf1b} by defining new coefficients $\tilde b_{i,i+1} = r_{n+1} + b_{i,i+1}$. In this way we obtain a pure fragmentation system that, however, is not mass conservative. Such systems were considered in  \cite{Cai1991, Smith2012d}. While mathematically they are equivalent to \eqref{gf1b}, physically they describe different models as in \eqref{gf1b} the death process is independent of fragmentation and in \cite{Cai1991} the mass loss is caused by the so-called explosive fragmentation. Also, the study of death--sedimentation process is of independent interest. \end{remark}

While the continuous form of \eqref{gf1b} has been well studied, both from theoretical, \cite{Cai1991, Huang1991, ArBa04, BaLa03, Bana12a}, and the numerical \cite{banasiak2014pseudospectral}, points of view, the discrete form has received much less attention, see \cite{Cai1991, Smith2011, Smith2012d}, where the authors considered a nonconservative fragmentation process resulting from the so-called random bond annihilation. In this paper, as in \cite{Smith2012d}, we shall use the substochastic semigroup theory, \cite{Banasiak2006a}, to prove the solvability of \eqref{gf1b} and investigate the properties of the solutions. The main results of the paper are the derivation of the conditions under which the solution semigroup is analytic and compact and hence the analysis of its long term-behaviour. In particular, we prove that if the sedimentation rate is at least as strong as the fragmentation rate and either is stronger than the death rate, the solution semigroup satisfies a spectral gap condition and consequently it has the asynchronous exponential growth property. This result gives a partial support to the observation in e.g. \cite{Ackleh1997} that size dependent sedimentation is a major factor in  the rapid clearance of material from the surface of the ocean.

\section{Theoretical analysis of the decay-fragmentation equation}
\label{genThe}
Following the general framework developed in \cite{Banasiak2006a}, we shall apply the theory of semigroups of operators.
Denote $X_0 =\ell^1$, the space of summable sequences. The analysis will be carried out in the subspace of $X_0$,
\[
X = \Bigl\{ \mb f :=(f_n)_{n=1}^\infty :
\| \mb f\| = \sum_{n=1}^{\infty} n | f_{n} |  < \infty \Bigr\}.
\]
The norm $\|\cdot\|$ has a simple physical interpretation --- for a given distribution of clusters
$\mb f = (f_{n})_{n=1}^{\infty}$, $\|f\|$ is the total mass of the system. For $\mb x, \mb y \in X$,
we denote $\mb x\mb y = (x_ny_n)_{n=1}^\infty$. For any $Y \subseteq X$, by $Y_+$,
we denote the set of all nonnegative sequences in $Y$.

The ACP associated with \eqref{gf1b} is given by
\begin{equation}
\label{acp1}
\frac{d \mb f}{d t} =  A \mb f +  B \mb f, \quad \mb f(0) = \mathring{\mb f}.
\end{equation}
The operators $(A, D(A))$ and $(B, D(B))$ are defined as restrictions of the expressions
\begin{equation*}
\begin{aligned}
[\mbb A \mb f]_{i} &= r_{i+1}f_{i+1} - (r_i + d_i + a_i)f_i, \qquad [\mbb B \mb f]_{i} = \sum_{j= i+1}^{\infty} a_{j} b_{i,j} f_{j}, \quad i \geq 1,
\end{aligned}
\end{equation*}
where, recall, $a_1=0$, on the domains, respectively $D(A)$ and $D(B)$ to be determined below. The main tool is
the Kato-Voigt Perturbation Theorem, \cite[Corollary 5.17]{Banasiak2006a}.

\subsection{The decay semigroup}
We begin by establishing the existence of a $C_0$-semigroup generated by a suitable realization of the decay operator $A$.
Consider the diagonal/off-diagonal splitting
$\mbb A = \mbb A_0 + \mbb A_1$ and the operators $(A_0,D(A_0))$ and $(A_1,D(A_1))$ being the restrictions of $\mbb A_0$ and  $\mbb A_1,$ defined as follows
\begin{align*}
&[ A_0\mb f]_i = - \theta_i f_i, \quad i\ge 1,\quad D(A_0) = \{f \in X: \boldsymbol{\theta} \mb f \in X  \},\\
&[ A_1\mb f]_i =   r_{i+1}f_{i+1}, \quad i\ge 1,\quad D(A_1) = \{f \in X: \mb{rf} \in X  \},
\end{align*}
 where $\theta_i = r_i + a_i + d_i$. It is  easy to see that $(A_0, D(A_0))$ generates a substochastic $C_0$-semigroup  in $X$.
Furthermore, $(A_1, D(A_1))$ is positive and $D(A_0)\subset D(A_1)$, so that the Kato--Voigt perturbation
theory applies. In particular, we have
\begin{theorem}\label{alt3}
The closure of $(A_0 + A_1, D(A_0))$ generates a substochastic $C_0$-semigroup
$\{S_{\overline{A_0+A_1}}(t)\}_{t\ge 0}$ in~$X$.
\end{theorem}
\begin{proof}
For $ \mb f \in D(A_0)_+$, we have
\begin{align*}
\sum_{n=1}^{\infty} n [(A_0 + A_1) \mb f]_{n} &= \sum_{n=1}^{\infty} n(r_{n+1}f_{n+1} -  r_n f_n ) - \sum_{n=1}^{\infty} n (a_n+d_n) f_n   \\
&= - \sum_{n=1}^{\infty} n\Bigl(a_n+d_n + \frac{r_n}{n}\Bigr) f_n =:-\sum_{n=1}^\infty n c_n f_n=: -c(\mb f) \le 0.
\end{align*}
Hence, there exists a smallest substochastic semigroup $\{S_{A}(t)\}_{t \geq 0}$ generated by an extension
$(K, D(K))$ of $A_{0} + A_1$. In fact $(K, D(K))$ is the closure of $A_0+A_1$. To see this, consider
the maximal extensions $\mathbb{A}_0$ and $\mathbb{A}_1$ of $A_0$ and $A_1$, respectively. Since $K \subset \mathbb{A}_0+\mathbb{A}_1$, see \cite[Theorem 6.20]{Banasiak2006a}, for
$ \mb f \in D(K)_+ $, we have
\begin{eqnarray}
\sum_{n=1}^{\infty} n [(\mathbb{A}_{0} + \mathbb{A}_{1})\mb  f]_{n} &=& \lim_{N \to \infty} \sum_{n=1}^{N} n \bigg (  - d_nf_n - a_nf_n - r_nf_n + r_{n+1}f_{n+1} \bigg ) \nonumber\\
&=& \lim_{N \to \infty} \bigg (- \sum_{n=1}^{N} n\Bigl(a_n+d_n + \frac{r_n}{n}\Bigr) f_n
 + Nr_{N+1}f_{N+1} \bigg ) \nonumber\\
&=& -\sum_{n=1}^\infty n c_n f_n +  \lim_{N \to \infty} Nr_{N+1}f_{N+1}  \ge - c(\mb f).
\label{KV1}
\end{eqnarray}
Hence $K = \overline{A_{0} + A_{1}}$ and  $\sum_{n=1}^\infty n [K\mb f]_n = -c(\mb f)$, \cite[Theorem 6.22]{Banasiak2006a}.
\end{proof}
In view of \cite[Theorem 2.3.4.]{Wong},  Theorem~\ref{alt3} implies that $\{S_K(t)\}_{t\ge 0}$ is unique in the sense this is the only semigroup whose generator is extensions of $(A_0+A_1, D(A_0))$.

The characterization  $K=\overline{A_0+A_1}$ is sharp but not easy to use. An alternative characterization of $K$ follows directly from \cite[Lemma 3.50 \& Proposition 3.52]{Banasiak2006a}.
Let $K_{\max} = \mathbb{A}_0+\mathbb{A}_1$ on
$$
D(K_{\max}) = \{\mb f \in X;\; \mathbb{A}_0\mb f+\mathbb{A}_1\mb f \in X\}.
$$
\begin{proposition}\label{maxdom}
$$
D(K) = D(K_{\max})
$$
if and only if there is no $\mb f \in X$ satisfying
\begin{equation}
\mathbb{A}_0\mb f+\mathbb{A}_1\mb f = \lambda \mb f, \quad \lambda>0.
\label{eigeq1}
\end{equation}
\end{proposition}
\begin{proof}
Equation (\ref{eigeq1}) takes the form
\begin{equation}
(\lambda + \theta_i )f_{i} -  r_{i+1}f_{i+1}  = 0, \quad i \geq 1.
\label{2.4}
\end{equation}
Formally, it can be solved by fixing $f_{\la,1}=(\la+\theta_1)^{-1} >0$ and recursively calculating
\begin{equation}
f_{\la, n} = \frac{f_{\la,1}}{\la+\theta_n}\prod\limits_{i=2}^n\frac{\la + \theta_i}{r_i}, \quad n\geq 2.
\label{2.4sol}
\end{equation}
Hence the solution space is at most 1 dimensional. Thus, by \cite[Proposition 3.52]{Banasiak2006a}, $K_{\max}$ is closed and \cite[Lemma 3.50]{Banasiak2006a} yields the decomposition
$$
D(K_{\max}) = D(K)\oplus Ker (\lambda I -K_{\max}).
$$
Hence the proposition follows.
\end{proof}
We will analyse (\ref{eigeq1}) later. Here we list the properties of $K$ that can be deduced from Theorem \ref{alt3}.
\begin{corollary}\label{cor2.2}
If $\mb f \in D(K)_+$, then
\begin{equation}
\sum\limits_{n=1}^\infty na_n f_n <\infty, \quad \sum\limits_{n=1}^\infty nd_n f_n <\infty, \quad \sum\limits_{n=1}^\infty r_n f_n <\infty
\label{finsums}
\end{equation}
and
\begin{equation}
\lim\limits_{n\to \infty} nr_nf_n =0.
\label{limcon}
\end{equation}
\end{corollary}
\begin{proof}
Properties (\ref{finsums}) follow from the fact that the functional $c$ extends to $D(K)$ by continuity (in the norm of $D(K)$) and to $D(K)_+$ by monotonic limits, \cite[Theorem 6.8]{Banasiak2006a}, as in the proof of \cite[Theorem 2.1]{Banasiak2012}. Eqn. (\ref{limcon}) follows from \cite[Theorem 6.13]{Banasiak2006a}. \end{proof}
For further consideration, we need one more result that is an obvious consequence of (\ref{finsums}) and the definition of $D(A_0)$.
\begin{corollary}
If there is $C>0$ such that
\begin{equation}
d_n+a_n \geq C r_n, \quad n \in \mbb N,
\label{condanal}
\end{equation}
then $D(K) = D(A_0)$ and $K=A_0+A_1$.
\label{coranal}
\end{corollary}
To provide a complete characterization of $D(K)$ in general case, we find an explicit formula for the resolvent.

\subsection{An alternative characterization of $K$}
Consider the formal resolvent equation for $\mbb A$:
\begin{equation}\label{dsg2}
(\lambda + \theta_i )u_{i} -  r_{i+1}u_{i+1}  = f_{i},
\end{equation}
where $\mb f = (f_n)_{n=1}^\infty. $
Solving \eqref{dsg2} for $u_i$,  we get the formal identity:
    \begin{equation}
  \label{dsg5}
  u_{i} = \frac{1}{\lambda + \theta_i} \sum_{n=i}^{\infty} f_n \prod_{j=i+1}^{n} \frac{r_{j}}{\lambda + \theta_{j}}, \quad  i\ge 1.
  \end{equation}
 It turns out that the operator $R_\lambda:X\to X$, defined by \eqref{dsg5}, is bounded.
\begin{lemma}\label{lem11}
For $\lambda > 0 $,  we have
$\|R_\lambda \| \leq \frac{1}{\lambda}$.
\end{lemma}
\begin{proof}
Let $\mb f \in X$. We apply the triangle inequality and change the order of summation
to obtain
\begin{equation}\label{eq14}
\begin{aligned}
\| R_\lambda\mb f \| &= \sum_{n=1}^{\infty} n \Bigl| \frac{1}{\lambda + \theta_n} \sum_{i=n}^{\infty} f_i \prod_{j=n+1}^{i} \frac{r_{j}}{\lambda + \theta_{j}} \Bigr|
\le \sum_{i=1}^{\infty} |f_{i}| \sum_{n=1}^{i} \frac{n}{\lambda + \theta_n} \prod_{j=n+1}^{i} \frac{r_{j}}{\lambda + \theta_{j}}\\
&  \le \frac{1}{\lambda} \sum_{i=1}^{\infty} |f_{i}| \sum_{n=1}^{i} n \bigg ( 1- \frac{\theta_n}{\lambda + \theta_n} \bigg )  \prod_{j=n+1}^{i} \frac{\theta_{j}}{\lambda + \theta_{j}} \\
&\leq \frac{1}{\lambda} \sum_{i=1}^{\infty} |f_{i}| \bigg[ \sum_{n=1}^{i} n \prod_{j=n+1}^{i} \frac{\theta_{j}}{\lambda + \theta_{j}} - \sum_{n=1}^{i} n \prod_{j=n}^{i} \frac{\theta_{j}}{\lambda + \theta_{j}}\bigg] \\
&= \frac{1}{\lambda} \sum_{i=1}^{\infty} |f_{i}| \bigg[ \sum_{n=1}^{i} n \prod_{j=n+1}^{i} \frac{\theta_{j}}{\lambda + \theta_{j}} - \sum_{n=0}^{i-1} (n+1) \prod_{j=n+1}^{i} \frac{\theta_{j}}{\lambda + \theta_{j}}\bigg] \\
&=  \frac{1}{\lambda} \sum_{i=1}^{\infty} |f_{i}| \bigg[ i - \sum_{n=1}^{i-1} \prod_{j=n+1}^{i} \frac{\theta_{j}}{\lambda + \theta_{j}}  - \prod_{j=1}^{i} \frac{\theta_{j}}{\lambda + \theta_{j}}\bigg]
\leq  \frac{1}{\lambda} \sum_{i=1}^{\infty} i |f_{i}|
\end{aligned}
\end{equation}
and the required estimate follows.
\end{proof}
Next, we show that $R_\lambda$, $\lambda>0$, is the resolvent of $A$ if the domian $D(A)$ is appropriately
chosen. To simplify our calculations, we let $\mb p = (p_n)^\infty_{n=1}$ with $p_n = a_n+d_n$,  and
$\triangle \mb f = (f_1, (f_{n-1}-f_{n})_{n\ge 2})$.
\begin{lemma}
\label{lem11b}
Operator $R_{\lambda}$, $\lambda > 0,$ maps $X$ into the set
\begin{equation}
S = \{ \mb f \in X: \mb{pf} \in X,  \triangle (\mb{rf}) \in X, \lim_{n\to\infty} r_n f_n = 0 \}.
\label{S}
\end{equation}
\end{lemma}
\begin{proof}
Our estimates are similar to those
used in formula \eqref{eq14}:
\begin{align*}
\| \mb pR_\lambda\mb f \| &=  \sum_{n=1}^{\infty} n \Bigl| \frac{(a_n + d_n)}{\lambda + \theta_n} \sum_{i=n}^{\infty} f_i \prod_{j=n+1}^{i} \frac{r_{j}}{\lambda + \theta_{j}}  \Bigr|\leq  \sum_{i=1}^{\infty} |f_i| \sum_{n=1}^{i} \frac{n (\theta_n - r_n)}{\lambda + \theta_n} \prod_{j=n+1}^{i} \frac{r_{j}}{\lambda + \theta_{j}} \\
&= \sum_{i=1}^{\infty} |f_i| \sum_{n=1}^{i} n \bigg [ \frac{-\lambda}{\lambda + \theta_n} + 1 - \frac{r_n}{\lambda + \theta_n} \bigg ] \prod_{j=i+1}^{i} \frac{r_{j}}{\lambda + \theta_{j}}\\
& \leq \sum_{i=1}^{\infty} |f_i| \sum_{n=1}^{i} n \bigg [ 1 - \frac{r_n}{\lambda + \theta_n} \bigg ] \prod_{j=i+1}^{i} \frac{r_{j}}{\lambda + \theta_{j}} \\
&= \sum_{i=1}^{\infty} |f_i| \bigg [ i + \sum_{n=1}^{i-1} n \prod_{j=n+1}^{i} \frac{r_{j}}{\lambda + \theta_{j}} - \sum_{n=1}^{i} n \prod_{j=n}^{i} \frac{r_{j}}{\lambda + \theta_{j}} \bigg ] \\
&= \sum_{i=1}^{\infty} |f_{i}| \bigg[ i - \sum_{n=1}^{i-1} \prod_{j=n+1}^{i} \frac{r_j}{\lambda+\theta_{j}}  - \prod_{j=1}^{i} \frac{r_j}{\lambda+\theta_{j}} \bigg]
\leq \sum_{i=1}^{\infty} i |f_i| = \|\mb f\|.
\end{align*}
Next, we have to estimate $\|\triangle (\mb r R_\lambda\mb f)\|$. For this, we observe that
\begin{align*}
[\triangle (\mb r R_\lambda\mb f)]_n &= r_{n+1}[R_\lambda\mb f]_{n+1} - r_n [R_\lambda\mb f]_n\\&=
(\lambda + p_n)[R_\lambda\mb f]_n + [ r_{n+1}[R_\lambda\mb f]_{n+1} - (\lambda + \theta_n)[R_\lambda\mb f]_n ] \\
&= (\lambda + p_n)[R_\lambda\mb f]_n  +  \frac{r_{n+1}}{\lambda + \theta_{n+1}} \sum_{i=n+1}^{\infty} f_i \prod_{j=n+2}^{i} \frac{r_{j}}{\lambda + \theta_{j}}  -  \sum_{i=n}^{\infty} f_i \prod_{j=n+1}^{i} \frac{r_{j}}{\lambda + \theta_{j}}  \\
&= (\lambda + p_n)[R_\lambda\mb f]_n  +  \sum_{i=n+1}^{\infty} f_i \prod_{j=n+1}^{i} \frac{r_{j}}{\lambda + \theta_{j}}  -  \sum_{i=n}^{\infty} f_i \prod_{j=n+1}^{i} \frac{r_{j}}{\lambda + \theta_{j}}  \\
&= (\lambda + p_n)[R_\lambda\mb f]_n - f_n.
\end{align*}
Consequently,
\begin{align*}
\| \triangle( \mb{rR_\lambda\mb f}) \|  &= \| (\lambda +\mb p)R_\lambda\mb f - \mb f \|  \leq \lambda \| R_\lambda\mb f \| + \|\mb pR_\lambda\mb f \| + \| \mb f \|\le  3 \| \mb f \|.
\end{align*}
Finally,
\[
\lim_{n\to\infty} |r_n[R_\lambda\mb f]_n| \le
\lim_{n\to \infty} \frac{r_n}{\lambda+\theta_n}\sum_{i=n}^\infty
|f_i| \prod_{j=n+1}^i\frac{r_j}{\lambda+\theta_j}
\le \lim_{n\to \infty} \sum_{i=n}^\infty |f_i| = 0,
\]
and the claim is proved.
\end{proof}

\begin{lemma}\label{lem12}
For $\lambda > 0$, the operator $R_\lambda$ is the resolvent of $(A, S),$ where $A= \mbb A|_{S}$.
\end{lemma}
\begin{proof}
First, we show that $R_\lambda$, $\lambda>0$ is the right inverse of $(I-A, S)$. Indeed,
$R_\lambda $ maps $X$ into $S$, hence $(\lambda I - A)R_\lambda: X \rightarrow X $
is well defined.
Further, for any $\mb v = R_\lambda \mb f$, $\mb f \in X,$ we have
\begin{align*}
[(\lambda I - A)R_\lambda\mb f]_{n} &= (\lambda + \theta_n) [R_\lambda\mb f]_{n} -  r_{n+1}[R_\lambda\mb f]_{n+1} \\
&= \sum_{i=n}^{\infty} f_i \prod_{j=n+1}^{i} \frac{r_{j}}{\lambda + \theta_{j}} - \sum_{i=n+1}^{\infty} f_{i} \prod_{j=n+1}^{i} \frac{r_{j}}{\lambda + \theta_{j}} \\
&= f_n + \sum_{i=n+1}^{\infty} f_i \prod_{j=n+1}^{i} \frac{r_{j}}{\lambda + \theta_{j}} - \sum_{i=n+1}^{\infty} f_{i} \prod_{j=n+1}^{i} \frac{r_{j}}{\lambda + \theta_{j}} = f_n.
\end{align*}

Next we show that the operator $R_\lambda, \lambda>0,$ is the left inverse of $\lambda I - A$. Indeed, for any
$\mb f \in S$, we have
\begin{align*}
[R_\lambda(\lambda I - A)\mb f]_n &=  \lim_{N\to\infty} \frac{1}{\lambda + \theta_n} \sum_{i=n}^{N}
((\lambda + \theta_{i}) f_{i} -  r_{i+1}f_{i+1}) \prod_{j=n+1}^{i} \frac{r_{j}}{\lambda + \theta_{j}} \\
&= \lim_{N\to\infty} \biggl(\sum_{i=n}^{N} f_i \prod_{j=n}^{i-1} \frac{r_{j+1}}{\lambda + \theta_{j}} -
\sum_{i=n+1}^{N+1} f_{i} \prod_{j=n}^{i-1} \frac{r_{j+1}}{\lambda + \theta_{j}}\biggr) \\
&= f_n - \lim_{N\to\infty} f_{N+1}r_{N+1} \prod_{j=n}^N\frac{r_{j}}{\lambda+\theta_j}  = f_{n},
\end{align*}
for
\[
\lim_{N\to\infty} |f_{N+1} r_{N+1}| \prod_{j=n}^N\frac{r_{j}}{\lambda+\theta_j}
\le \frac{1}{\lambda}\lim_{N\to\infty} |r_N f_N| = 0,
\]
when $\mb f \in S$. We conclude that $R_\lambda$, $\lambda>0,$ is the resolvent of $(A,S)$.
\end{proof}
We can write $S=D(A)$. The following result gives an explicit characterization of  $\overline{A_0+A_1}$.
\begin{theorem}\label{th2.5}
The operator $(A, D(A))$ generates a substochastic $C_0$-semigroup $\{S_{A}(t)\}_{t \geq 0}$ on $X$.
Furthermore, $K=\overline{A_0+A_1} = A$.
\end{theorem}
\begin{proof}
The first assertion follows from Lemmas~\ref{lem11},  \ref{lem12},  the classical Hille-Yosida
theorem \cite{Pazy1983} and positivity of the resolvent.
The second statement is a  by product of Theorem~\ref{alt3} - it follows by the uniqueness of the semigroup.
\end{proof}
We observe that there is an apparent inconsistency between Theorem \ref{th2.5} and Corollary \ref{cor2.2}, as conditions (\ref{finsums}) and (\ref{limcon}) are different than that defining $D(A)$. We can, however, directly prove that, indeed, if $\mb f \in D(A)$, then it satisfies (\ref{finsums}) and (\ref{limcon}).  
\begin{proposition} Define
\begin{equation}
S'=  \{\mb f \in X\;: \mb {pf} \in X, \mb{ r f}\in X_0, \lim\limits_{n\to \infty} nr_nf_n =0\}.
\label{con1}
\end{equation}
Then $D(A) \subset S'$.\label{prop2.2}
\end{proposition}
\begin{proof}
Let $\mb f \in S$. In general, the positive and negative parts of $\mb f$, $\mb f_\pm,$ do not belong to $S$.  However, since $S=D(A)$ is the domain of a resolvent positive operator, we can write $\mb f$ as a difference of nonnegative elements $\mb f = \mb f^+-\mb f^-$ (not necessarily the positive and negative part of $\mb f$). Indeed, for some $\mb g \in X$ we have $$\mb f = R(\lambda, A)\mb g = R(\lambda,A)\mb g_+-R(\lambda,A)\mb g_-=:\mb f^+-\mb f^-, $$ where $\mb g = \mb g_+-\mb g_-$ is the decomposition of $\mb g$ into its positive and negative parts in $X$.  Thus we can assume $\mb f \in S_+$.  Since  $\triangle \mb{rf} \in X$, also $\sum_{n=1}^\infty [\triangle \mb{rf}]_n$ converges. However,
\begin{eqnarray*}
\sum\limits_{n=1}^\infty [\triangle (\mb{rf})]_n &=& \lim\limits_{N\to \infty}\sum\limits_{n=1}^N [\triangle (\mb{rf})]_n = 2r_1f_1 - \lim\limits_{N\to \infty}r_Nf_N = 2r_1f_1,
\end{eqnarray*}
by the last condition in (\ref{S}).
Again by $\triangle \mb{rf} \in X,$ the terms in the sum can be rearranged
\begin{eqnarray*}
\sum\limits_{n=1}^\infty n[\triangle (\mb{rf})]_n &=& r_1f_1 + 2(r_1f_1-r_2f_2) + 3(r_2f_2-r_3f_3)+\ldots \\
&=&  \sum\limits_{n=1}^\infty [\triangle (\mb{rf})]_n + \sum\limits_{n=2}^\infty (r_{n-1}f_{n-1}-r_nf_n) + \sum\limits_{n=3}^\infty  (r_{n-1}f_{n-1}-r_nf_n) +\ldots\\
&=& 2r_1f_1 + \sum\limits_{n=1}^\infty r_nf_n,
\end{eqnarray*}
 hence  $\mb{rf}\in X_0$. To complete the proof, we note that also
\begin{equation}
\sum\limits_{n=1}^\infty n[\triangle (\mb{rf})]_n = \lim\limits_{N\to \infty}\sum\limits_{n=1}^N [\triangle (\mb{rf})]_n = 2r_1f_1  +\lim\limits_{N\to \infty}\left(\sum\limits_{n=1}^{N-1} r_nf_n - Nr_Nf_N\right).
\label{boo}
\end{equation}
This shows that $\lim_{N\to \infty}Nr_Nf_N = c\ge 0$ exists. If  $c> 0$, then $r_Nf_N \geq \frac{c'}{N}>0$
for $c'>0$ and sufficiently large $N$ which contradicts $\mb {rf} \in X_0$. Hence $S\subset S'$. 
\end{proof}
The last aspect to be clarified is the relation of $D(A)$ to the domain $D(K_{\max})$, explicitly given as
\begin{equation}
D(K_{\max}) =\{\mb f\in X\;: \sum\limits_{n=1}^\infty n|a_nf_n+d_nf_n + r_nf_n-r_{n+1}f_{n+1}| <\infty\}.
\label{dkmax}
\end{equation}
Proposition \ref{maxdom} asserts that
$$
D(K_{\max}) = S,
$$
provided (\ref{eigeq1}) has no solutions in $X$. The latter  is related to the last condition in (\ref{S}),
\begin{equation}\label{eq2.7}
\lim_{n\to\infty} r_nf_n = 0,
\end{equation}
that, in general, cannot be discarded. We observe that the solution $\mb f_\la$ to (\ref{2.4}), given by (\ref{2.4sol}), belongs to $X$ if and only if
$$
\sum\limits_{n=2}^\infty \frac{n}{\la+\theta_n}\prod\limits_{i=2}^n\frac{\la+\theta_i}{r_i}<\infty.
$$
In particular, if
\begin{subequations}
\begin{equation}
\label{eq112a}
\prod_{j=2}^{\infty} \frac{\lambda + \theta_{j}}{r_{j}}<\infty ,
\end{equation}
\begin{equation}
\label{eq112b}
\sum_{n=1}^{\infty} \frac{n}{\lambda + \theta_n} < \infty,
\end{equation}
\end{subequations}
then $\mb f_\la \in X$.
The relation between $D(K_{\max})$ and $D(A)$ is then settled by the following result.
\begin{lemma}
If one of the conditions \eqref{eq112a} or \eqref{eq112b} is not satisfied and $\mb {pf}, \triangle (\mb {rf}) \in X$,
then \eqref{eq2.7} holds.
\end{lemma}
\begin{proof}
Note that in view of the condition $\triangle (\mb {rf}) \in X$, $\lim_{n\to\infty} r_nf_n$  exists and is finite.
Assume $\lim_{n\to\infty} r_nf_n \ne 0$.
Then $|r_nf_n| \geq c$ for some  $c>0$ and all $n$ sufficiently large.
If \eqref{eq112a} does not hold,  then the series
$\sum_{j=1}^{\infty} \frac{\lambda + p_j}{r_j}$  diverges. However, for large values of $n$, we have
\[
\sum_{j=n}^{n+k} \frac{\lambda + p_j}{r_j} \leq  \frac{1}{c} \sum_{j=n}^{n+k} (\lambda + p_j)
|f_j| \le \frac{1}{c} \sum_{j=n}^{\infty} (\lambda + p_j) |f_j| < \infty,
\]	
a contradiction.

Similarly, if   \eqref{eq112b} does not hold, then
\[
\sum_{n=1}^{N} \frac{n}{\lambda + \theta_n} \le \frac{1}{c} \sum_{n=1}^{N} n |f_n|\frac{r_n}{\lambda + \theta_n} \le \frac{1}{c} \sum_{n=1}^{\infty} n  |f_n| < \infty.
\]
Hence, \eqref{eq2.7} holds.
\end{proof}
\begin{example}
Summarizing, if there are no solutions to (\ref{2.4}) in $X$, then one of the conditions  \eqref{eq112a} or \eqref{eq112b} is not satisfied and hence
\begin{equation}
D(A) = D(K_{\max}) = \{ \mb f \in X: \mb{pf} \in X,  \triangle (\mb{rf}) \in X \}.
\label{dadkmax}
\end{equation}
This is a typical case. We observe that for (\ref{eq112a}) to hold, $\sum_{j=1}^{\infty} \frac{\lambda + p_j}{r_j}$ must converge and this requires $r_j$ to diverge to infinity much faster that $p_j$. Then, in (\ref{eq112b}), $r_j$ would be a dominant term in $\theta_j$. For instance, if $r_j = j^p$, $p\leq 2$, or either $a_j$ or $d_j$ grows faster than $r_j/j$,  then (\ref{dadkmax}) holds. We observe that the latter condition is satisfied if, in particular, the assumption of Corollary \ref{coranal} is satisfied.

We also note that if (\ref{dadkmax}) holds, then we can consider $A$ as the sum
$$
A = L+D
$$
of the diagonal operator $L\mb f = -\mb{pf}$, corresponding to the loss term in the fragmentation and the sedimentation, and the operator $D\mb f = (r_{n+1}f_{n+1}-r_nf_n)_{n=1}^\infty$, corresponding to the death of particles, mentioned in Introduction. Such a representation makes more sense as it separates two independent processes driving the evolution of the system.
\end{example}

\subsection{The decay-fragmentation semigroup}
Now we turn to the complete model \eqref{acp1}. As in any fragmentation model, it can be shown that $\mbb B$ is finite and nonnegative on $D(B) = D(A_0)\supset D(A)$, hence the operator $(A+B, D(A))$ is well defined.  Again, we shall apply the Kato-Voigt theorem \cite{Banasiak2001, Banasiak2006a} to show that there exist a smallest extension $G$ of $A + B$ that generates a substochastic semigroup
$\{S_{G}(t)\}_{t \geq 0}$ on $X$.
\begin{theorem}
\label{df1}
Let $X$, $A$ and $B$ be as defined earlier. Then the closure $G = \overline{A + B}$ generates a
substochastic semigroup $\{S_{G}(t)\}_{t \geq 0}$ on $X$. If, in addition, for some $\lambda>0$
\begin{equation}
	\label{eq2.8}
	\limsup_{n \to \infty} \frac{a_n}{d_n}  < \infty,
\end{equation}
then $D(G) = D(A)$.
\end{theorem}
\begin{proof}
 In the previous section we have proved  that $A$ generates the substochastic semigroup
$\{S_{A}(t)\}_{t \geq 0}$ on $X$.  It remains to verify that $A+B$ satisfies the last condition of the Kato-Voigt theorem.

Let $ \mb f \in D(A)$. By Proposition \ref{prop2.2}
\[
\lim_{n\to\infty} nr_nf_n = 0.
\]
With the aid of the last identity, (\ref{equ12}) and remembering that $a_1=0$, for $ \mb f \in D(A)_+$ we have, as in Theorem \ref{alt3},
\begin{align*}
\sum_{n=1}^{\infty} n [A \mb f + B \mb f]_{n} &= \sum_{n=1}^{\infty} n \bigg ( r_{n+1}f_{n+1} - (r_n + d_n + a_n)f_n + \sum_{j= n+1}^{\infty} a_{j} b_{n,j} f_{j} \bigg ) \\
%&= \sum_{n=1}^{\infty} n(r_{n+1}f_{n+1} -  r_n f_n ) - \sum_{n=1}^{\infty} n (a_n+d_n) f_n
 %+ \sum_{j=1}^{\infty} a_{j} f_{j} ( \sum_{n=1}^{j-1} n b_{n,j} ) \\
%&= \lim_{N\to\infty}\sum_{n=1}^{N} n (r_{n+1}f_{n+1} - r_n f_n)  - \sum_{n=1}^{\infty} n d_n f_n   \\
%&= -\lim_{N\to\infty} \Bigl(\sum_{n=1}^{N} r_nf_n -  Nr_{N+1}f_{N+1} \Bigr)  - \sum_{n=1}^{\infty} n d_n f_n \\
&= - \sum_{n=1}^{\infty} (r_n + n d_n) f_n  :=-\sum_{n=1}^\infty c_n f_n:= -c(\mb f) \le 0,
\end{align*}
 where the fragmentation part vanishes due to the conservativeness of the fragmentation process.   Hence, there exists a smallest substochastic semigroup, $\{S_{G}(t)\}_{t \geq 0}$, generated by an extension $(G, D(G))$ of $A + B$.

 Also as in Theorem \ref{alt3}, for the extensions $\mathbb{A}$ and $\mathbb{B}$  of $A$ and $B$ and  $ \mb f \in D(G)_+ $,
we have
\begin{align*}
\sum_{n=1}^{\infty} n [\mathbb{A} \mb f + \mathbb{B} \mb f]_{n} &= \lim_{N \to \infty} \sum_{n=1}^{N} n \bigg ( r_{n+1}f_{n+1} - (r_n + d_n + a_n)f_n + \sum_{j= n+1}^{\infty} a_{j} b_{n,j} f_{j} \bigg ) \\
%&= \lim_{N \to \infty} \bigg ( - \sum_{n=1}^{N} c_n f_n
%+ \sum_{n=N+1}^{\infty} a_nf_n \sum_{j=1}^{N} j b_{j,n} + Nr_{N+1}f_{N+1} \bigg ) \\
& = - c(\mb f) + \lim_{N \to \infty} \bigg ( \sum_{n=N+1}^{\infty} a_nf_n \sum_{j=1}^{N} j b_{j,n} + Nr_{N+1}f_{N+1} \bigg )\geq -c(\mb f).
\end{align*}
Hence $G = \overline{A+B}$.

To prove the last statement we argue as in Corollary \ref{cor2.2}. Since $G = \overline{A+B}$, we have
$$
\sum_{n=1}^{\infty} n [G \mb f]_{n} = -c(\mb f),
$$
for any $\mb f \in D(G)$. However, as shown in the proof of Proposition \ref{prop2.2}, it is sufficient to consider $\mb f \in D(G)_+$. Hence  $$
\lim_{N \to \infty} \bigg ( \sum_{n=N+1}^{\infty} a_nf_n \sum_{j=1}^{N} j b_{j,n} + Nr_{N+1}f_{N+1} \bigg )=0,
$$
yielding \[
\lim_{N\to\infty} Nr_Nf_N = 0, \quad \mb f \in D(G)_+,
\]
as both terms in the limit are nonnegative. Furthermore, since $c$ extends to $D(G)_+$, we have $\mb {rf}\in X_0$ and $\mb {df}\in X$.
Next, assuming that \eqref{eq2.8} holds, for $\mb f \in D(G)_+$, we have
\begin{equation*}
\sum_{n=1}^{\infty} n p_n f_n \le
C\bigg ( \sum_{n=1}^{\infty} (r_n + n d_n) f_n   \bigg )  = C c(\mb f),
\end{equation*}
for some constant $C$, so $\mb {pf} \in X$.
Hence, $D(A) =S'  \supseteq D(G)$,
and so $D(G) = D(A)$.
\end{proof}

\subsection{Analyticity and compactness}
In this section, we show that, under additional assumptions on the coefficients, the decay-fragmentation semigroup  $\{S_G(t)\}_{t\ge 0}$ is analytic and compact.  We use the following result from \cite{Arendt1991}.
\begin{theorem}\label{the3.2}
Let $X$ be a Banach lattice, $(A, D(A))$ be a generator of a positive analytic semigroup and
$(B, D(A))$ be a positive operator. Assume also that $(\lambda I - (A+B), D(A))$ has a nonnegative
inverse for some $\lambda>s(A)$. Then $(A+B, D(A))$ generates a positive analytic semigroup.
\end{theorem}

\begin{theorem}\label{the3.3}
Assume that \eqref{condanal} and \eqref{eq2.8} are satisfied. Then the semigroup $\{S_{G}(t)\}_{t \geq 0}$, generated by $(A + B, D(A))$ is analytic on $X$.
In addition, if
\begin{equation}
\liminf_{n\to\infty} \theta_n = \infty,
\label{condcom}
\end{equation}
then $\{S_{G}(t)\}_{t \geq 0}$ is immediately compact.
\end{theorem}
\begin{proof}
 Since (\ref{eq2.8}) is satisfied, Theorem \ref{df1} ensures that $(A+B, D(A))$ generates a positive $C_0$-semigroup of contractions in $X$.
Consequently, $R(\lambda, A+B)$ is positive for some $\lambda > 0$.
It was shown earlier that $(A, D(A))$ generates a substochastic semigroup in $X$. On the other hand, Corollary \ref{coranal} states that under (\ref{condanal}), we have
$A = A_0+A_1$, hence $A = A_0+A_1+B$ on $D(A)=D(A_0)$. Since $A_0$ is diagonal, it is clear that it generates an analytic semigroup.  Thus, Theorem \ref{the3.2} gives the analyticity of $\{S_G(t)\}_{t\geq 0}$ on account of $A_1+B$ being positive.

 Since $\{S_{G}(t)\}_{t \geq 0}$ is analytic, it is immediately uniformly continuous. Thus it remains to show that $R(\lambda, A+B)$ is compact for some
$\lambda>s(A)$, as then the result follows from \cite[Theorem 2.3.3]{Pazy1983}.

By virtue of Theorem~\ref{the3.2} and \cite[Theorem 4.3]{Banasiak2006a},  $I  - (A_1+B)R(\lambda, A_0)$ is invertible and
\[
R(\lambda, A_0+A_1+B) = R(\lambda, A_0) [I - (A_1+B)R(\lambda, A_0) ]^{-1}.
\]
In view of the last identity, it suffices to show that $R(\lambda, A_0)$ is compact for some $\lambda>0$.
For each $\mb f\in X$ with $\|\mb f\|\le 1$, we have $\|R(\lambda, A_0)\mb f\|\le \tfrac{1}{\lambda}$ and
\[
\sum_{n = n_0}^\infty \bigl|[R(\lambda,A_0) \mb f ]_{n}\bigr|
\le \sup\limits_{n\ge n_0} \frac{1}{\la +\theta_n} \sum_{n = n_0}^\infty n|f_n|,\]
If \eqref{condcom} holds, we have
\[
\lim_{n_0\to\infty} \sup_{n\ge n_0} \frac{1}{\la +\theta_n} = \frac{1}{\la +\liminf_{n\to \infty}\theta_n} = 0.
\]
Hence the image of the unit ball
$B=\{\mb f\in X : \|f\| \le 1\}$ under $R(\la, A_0)$ is bounded and uniformly summable  and therefore it is precompact, see \cite[IV.13.3]{DS}.
\end{proof}

\section{Long time behaviour of the decay-fragmentation semigroup}
Let us introduce the dual space to $X=\ell_1^1$, \[
X^{*} = \Bigl\{ \mb f^{*} :
\| \mb f^{*}\|_{X^{*}} := \sup_{k \ge 1} \frac{1}{k} | f_{k} |  < \infty \Bigr\}
\]
with the duality pairing
\[
\langle \mb f^{*}, \mb f\rangle := \sum_{i=1}^{\infty} f^{*}_{i} f_i,\quad \mb f \in X, \mb f^*\in X^*.
\]
The following result is analogous to classical theorems on asynchronous exponential growth, see e.g. \cite[Theorem VI.3.5]{Engel2006} since, however, \sem{S_G} is not irreducible, it requires a separate proof.
\begin{theorem}\label{th3.1}
Let the decay-fragmentation semigroup $\sem{S_{G}}$ satisfies conditions  \eqref{condanal}, \eqref{eq2.8} and \eqref{condcom} and let
\begin{equation}
\la_1 := -\min\limits_{n \in \mbb N}\theta_n
\label{eigeneq}
\end{equation}
be the strict minimum of the sequence $(\theta_n)_{n=1}^\infty$. Then there exist constants $\epsilon > 0,$ $M \ge 1$ and $\mb e\in X, \mb e^*\in X^*$ such that for any $\mb f \in X$
\begin{equation}
\label{lt10}
\| e^{-\lambda_{1} t} S_G (t)\mb f  - \langle \mb e^*,\mb f\rangle \mb e \| \le M e^{-\epsilon t}.
\end{equation}
\end{theorem}
\begin{proof}
We follow the proof of \cite[Theorem 4.3]{BL2012} with some modifications. Under the adopted assumptions,  $R(\la, G)$ is compact, hence $\sigma(G)$ is countable and consists of poles of $R(\la, G)$ of finite algebraic multiplicity, \cite[Corollary V.3.2]{Engel2000}. Thus, in particular, the essential radius of the semigroup and its essential growth rate satisfy, respectively, $r_{ess} (S_G(t)) = 0$ and $\omega_{ess}(G) = - \infty.$ Hence, any half-plane $\{\Re \la > a\;:a>-\infty\}$ contains only a finite number of eigenvalues of $G$.   Since, in addition, \sem{S_G} is positive, it follows that the peripheral spectrum of $G$ is additively cyclic and, being finite, consists of the single point $s(G)$  that is the dominant eigenvalue of $G$, see e.g.   \cite[Theorems 48 \& 49]{BanCIME2008}.

To prove that $\la_1 = s(G)$, first we observe that, by \eqref{condcom}, the infimum of $\theta_n$ is attained.   Let $P_N$ be the projection operator on $X$ defined by
\begin{equation}
P_{N} u = (u_1, u_2, .. ., u_N, 0, . . .)  \quad \text{for fixed} \quad N \in \mathbb{N}.\label{PN}
\end{equation}
The key observation is that the space $P_NX$ is invariant under $G =A+B$ (and thus under \sem{S_G})  for each $N$ -- this follows from the upper triangular structure of $A+B$. Denoting $G_N = G|_{P_NX}$, we have $\sigma(G_N) = \{-\theta_1, \ldots,-\theta_N\}$ and $s(G)\geq  \la_1$ as, due to the invariance, any eigenvalue of $G_N$ is an eigenvalue of $G$. On the other hand, assume $s:=s(G)>  \la_1$. Then  $s$ is an  isolated eigenvalue of $G$ and hence there is a decomposition $X = X_{1}\oplus X_{2},$ where $X_{1}$ is the spectral subspace corresponding to $s,$ whose dimension is at least 1, while $X_{2}$ is a closed complementary subspace invariant under $G$ on which $s I-  G$ is invertible. In particular, $(s I-  G)X \subset X_{2}$. On the other hand, since $s \notin \sigma(G_N)$ for any $N$
$$(s I-  G)X \supset (s I-  G)\bigcup\limits_{N=2}^\infty P_NX  = \bigcup\limits_{N=2}^\infty P_NX
$$
and the latter is dense in $X$. This contradiction shows $s(G)=  \la_{1}$.

By  \cite[Corollary V.3.2]{Engel2000} we can write
\begin{equation}
\label{lt11}
S_G(t) = S_{G}^{1} (t) + R(t),
\end{equation}
where there is $\epsilon > 0$ and  $M\geq 1$ such that
\begin{equation}
\label{lt12}
\| R(t)\| \le M_{\epsilon} e^{(\la_1-\epsilon) t} \quad t \ge 0.
\end{equation}
The operator $S_{G}^{1}$ has finite rank and hence is given by
\begin{equation}
\label{lt13}
S_{G}^{1} (t) =  \bigg ( e^{\lambda_1 t} \sum_{j=0}^{k-1} \frac{t^j}{j!} \bigg ) \Pi,
\end{equation}
where $k$ is the order of the pole $\lambda_1 = s(G)$ and  $\Pi$ is the corresponding spectral projection onto the subspace of $X$ whose dimension equals the algebraic multiplicity $m$ of $\la_1$.  We shall prove that  $k=m=1$.

The semigroup $(S_{G}^{1} (t))_{t\geq 0}$ is  finite-dimensional and it is generated by an operator which has $\lambda_1 = s(G)$ as the only eigenvalue. Since, by the previous part of the proof, $\|e^{\la_1 t} S_G(t)\| \leq 1$, also $\|e^{\la_1t}S_{G}^{1} (t)\|\leq 1$  and thus $k=1$, for otherwise $\| e^{\la_1 t}S_{G}^{1} (t) \| $ would grow polynomially in $t$. Using the inequality $m_g + k - 1 \le m \le m_g  k$, where $m_g$ is the geometric multiplicity, see e.g. \cite[p. 19]{BanCIME2008}, for $k=1$, we find
$m_g  = m$ which shows that to prove $m=1$, it is sufficient to show $m_g =1$.

We prove that $m_g =1$ by examining the adjoint operator $G^{*}$.  Arguing as in  \cite[Theorem 4.3]{BL2012}, we find that the operator $G^{*}$ defined by
\begin{equation}
\label{lt17}
\begin{aligned}
(G^{*} \mb f^{*})_{1} &:= -\theta_1 f_{1}^{*}, \\
(G^{*} \mb f^{*})_{j} &:= -\theta_j f_{j}^{*} + r_{j}f_{j-1}^{*} + a_j \sum_{i=1}^{j-1} b_{i,j} f_{i}^{*}, \quad j=2,3,\ldots,
\end{aligned}
\end{equation}
with the domain
\begin{equation}
\label{lt18}
D(G^{*}) := \bigg \{\mb f^{*} \in X^{*} : \sup_{j \ge 2} \frac{1}{j}\left| -\theta_j f_{j}^{*} + r_{j}f_{j-1}^{*} + a_j \sum_{i=1}^{j-1} b_{i,j} f_{i}^{*} \right| <  \infty \bigg \},
\end{equation}
is the adjoint of $G$.

Suppose that $\mb e^{*} = (e_1^*,e_2^*,\dots)$ is an eigenvector of $G^{*}$ corresponding to the eigenvalue $\lambda_1 = s(G)$ and assume $\la_1 = -\theta_{N_0}$. Then we see that $e_i^* = 0$ for $1\leq i\leq N_0-1$, the $N_0$-th equation is satisfied irrespectively of $e^*_{N_0}$, so $e^*_{N_0}$ can be chosen in an arbitrary way and then $e^*_{N}$ for $N>N_0$ can be recursively evaluated in a unique way (for a given $e^*_{N_0}$) as
\begin{equation}
e^*_N = \frac{1}{\theta_N - \theta_{N_0}} \left(r_N e^*_{N-1} + a_N\sum\limits_{j=N_0}^{N-1} b_{j,N}e^*_j\right).
\label{e*}
\end{equation}
Therefore, the geometric multiplicity of $\lambda_1 = s(G)$ is at most 1. Hence, $\lambda_1 = s(G)$  is a simple dominating eigenvalue of $G$. To complete the proof, let us find the explicit form of the spectral projection $\Pi$. For this we find the eigenvector $\mb e = (e_1,e_2,\ldots)$ of $G$ belonging to $\la_1=-\theta_{N_0}$. We observe that $\mb e$ satisfies
\begin{eqnarray*}
-\theta_{N_0} e_1 &=& -\theta_1 e_1 +r_2 e_2 + \sum_{j= 2}^{\infty} a_{j} b_{i,j} e_{j},\\
\vdots&=& \vdots ,\\
-\theta_{N_0} e_{N_0-1} &=& -\theta_{N_0-1} e_{N_0-1} +r_{N_0} e_{N_0} + \sum_{j= N_0}^{\infty} a_{j} b_{i,j} e_{j},\\
 0&=& r_{N_0+1}e_{N_0+1}  + \sum_{j= N_0+1}^{\infty} a_{j} b_{i,j} e_{j},\\
- \theta_{N_0} e_{N_0+1} &=& -\theta_{N_0+1} e_{N_0+1} +r_{N_0+2} e_{N_0+2} + \sum_{j= N_0+2}^{\infty} a_{j} b_{i,j} e_{j},\\
 \vdots&=&\vdots.
\end{eqnarray*}
We see that the   equations for $N\geq N_0+1$ decouple from the system and are solved by $e_N = 0, N\geq N_0+1.$  Then the $N_0$-th equation is trivially satisfied and the remaining $N_0-1$ equations form an upper triangular system with $N_0$ unknowns that can be solved recursively setting $e_{N_0} =1,$
\begin{equation}\label{eq3.10}
e_N = \frac{1}{\theta_N-\theta_{N_0}}\left(r_{N+1} e_{N+1} +   \sum_{j= N+1}^{N_0} a_{j} b_{i,j} e_{j}\right), \quad 1\leq N \leq N_0-1.
\end{equation}
 Choosing $\mb e^*$ with $e^*_{N_0}=1$, we obtain $\langle \mb e^*,\mb e\rangle =1,$ hence, by using the standard linear algebra formula on $P_NX$  and passing to the limit, we obtain the spectral projection
 $$
\Pi \mb f = \bigg (\sum_{k=1}^{\infty} e_k^{*} f_k \bigg ) \mb e = \langle\mb e^{*}, \mb f\rangle \mb e
$$
and, in view of \eqref{lt13},
\begin{equation}
\label{lt28}
S_{G}^{1} (t)\mb f = e^{\lambda_1 t}\Pi \mb f  = e^{\lambda_1 t} \langle\mb e^{*}, \mb f\rangle \mb e.
\end{equation}
The last formula combined with \eqref{lt12} yields \eqref{lt10}.
\end{proof}

\section{Numerical simulations}
In this section, we provide numerical study of the decay-fragmentation equation \eqref{gf1b}. Note that if the initial condition satisfies $u_n(0) = 0$, $n \ge N+1$, the complete model cannot generate clusters of sizes greater than $N$ and the dynamics is essentially finite dimensional. In view of this, we replace the complete model \eqref{gf1b} with its truncated version
\begin{equation}
\begin{aligned}
\label{gf3a}
&\frac{d u_{i}^N}{d t} = r_{i+1}u_{i+1}^N - r_{i}u_{i}^N - d_{i}u_{i}^N -  a_{i}u_{i}^N + \sum_{j= i+1}^{N} a_{j} b_{i,j} u_{j}^N, \quad i \geq 1, \\
&u_{i}(0) = u_{i}^0, \quad \text{for} \quad i = 1, \ldots, N.
\end{aligned}
\end{equation}
Problem \eqref{gf3a} is just a system of linear ODEs that can be integrated numerically. In our simulations, we employ the \verb"ode23tb" MATLAB solver. The latter is an implementation of TR-BDF2, an implicit Runge-Kutta formula with a trapezoidal rule step in the first stage and backward differentiation formula of order two in the second stage. In all our simulations, we set the time interval to be $[0, 20]$ and let $N=64$.

\begin{figure}[t]
\centering
\includegraphics{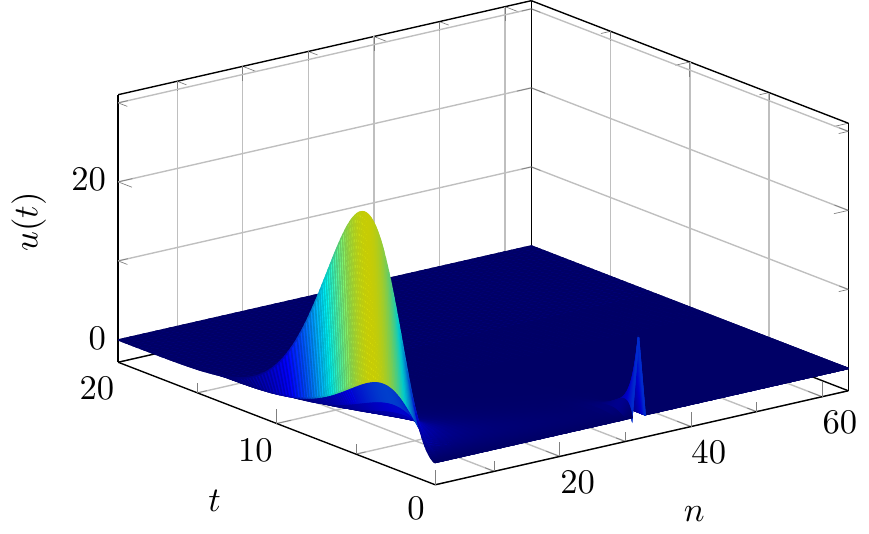}\\[24pt]
\includegraphics{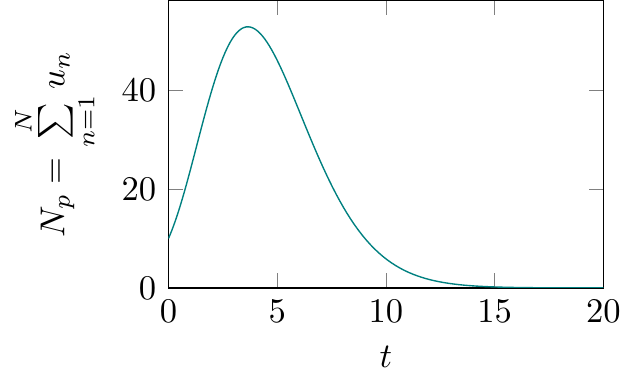} \hspace{3mm}
\includegraphics{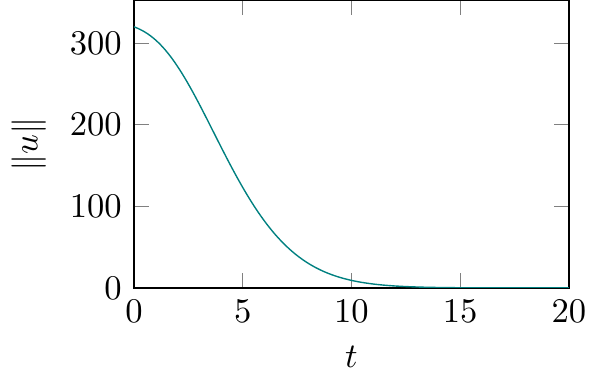}
\caption{Evolution of the decay-fragmentation model \eqref{gf1b} with constant transport and fragmentation rates (top);
evolution of total number of particles (bottom left); evolution of the total mass (bottom right).}\label{fig1}
\end{figure}

\subsection{Constant decay and fragmentation rates}

In our first example, we set $r_i = 1$, $d_i = 0$, $i\ge 1$; $a_1 = 0$ and $a_i=1$, $i\ge 2$;
and $b_{i,j} = \frac{2}{j-1}$, $i\ge 1$, $j\ge i+1$. The scenario falls in the scope of general
Theorem~\ref{df1} and models the decay-fragmentation process with constant decay and fragmentation
rates and no death. As the initial condition, we take the monodisperse distribution
\begin{align*}
u_{n}(0) = 10\delta_{n,32}, \quad 1\le n\le 64,
\end{align*}
where $\delta_{n,m}$ is the Kronecker delta.

\begin{figure}[t]
	\centering
	\includegraphics{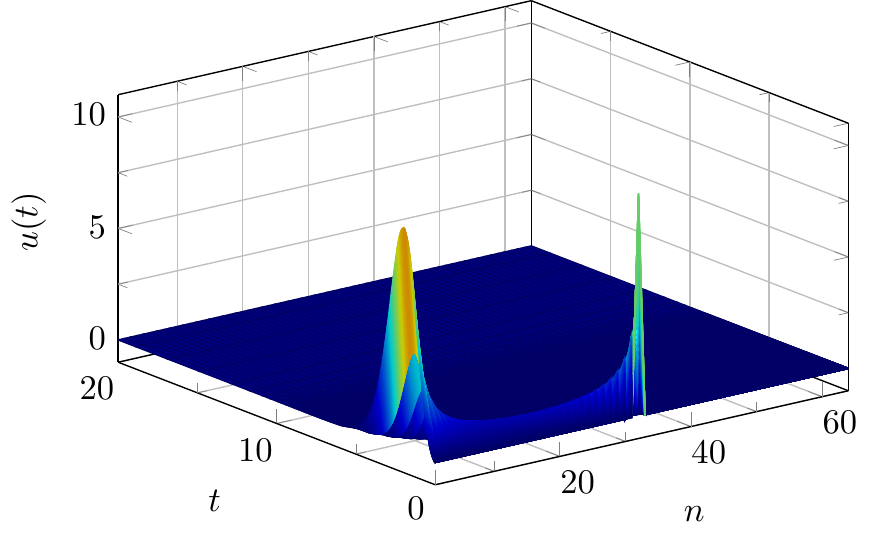}\\[24pt]
	\includegraphics{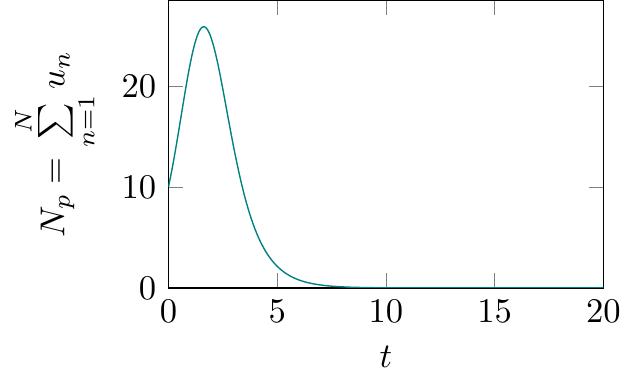} \hspace{3mm}
	\includegraphics{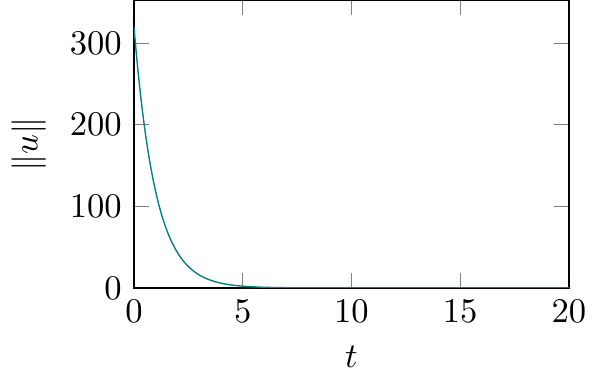}
	\caption{Evolution of the decay-fragmentation model \eqref{gf1b} with linear transport and constant fragmentation rates (top);
		evolution of total number of particles (bottom left); evolution of the total mass (bottom right).}\label{fig2}
\end{figure}

The dynamics of the model is shown in the top diagram of Fig.~\ref{fig1}.
As expected, when time increases no clusters of size greater than $10$ appears. Further, due to the fragmentation process the total number of aggregates of size $10$ is steadily decreasing while smaller clusters appear in the system. Due to the transport process, the total number of particles gradually decays and becomes almost
negligible when $t$ approaches the terminal time $T=20$.

Further illustration is provided by the two bottom diagrams in Fig.~\ref{fig1}. The left-bottom diagram shows the evolution of the total number of particles $N_p = \sum_{n=1}^{N} u_n$. The quantity is increasing initially due to the fragmentation process, and then after some transition time, decreases steadily to zero due to the transport (decay) process.

Note that in the model \eqref{gf1b}, the fragmentation process is conservative and the mass leakage is solely due to the decay. The latter process is monotone, in the sense that the total mass of a system in a pure decay equation shall decrease monotonically (see Theorem~\ref{alt3}). Theoretically, in \eqref{gf1b}, we expect similar mass dynamics as in the pure decay equation. This is confirmed in the right-bottom diagram of Fig.~\ref{fig1}.

\begin{figure}[t]
	\centering
	\includegraphics{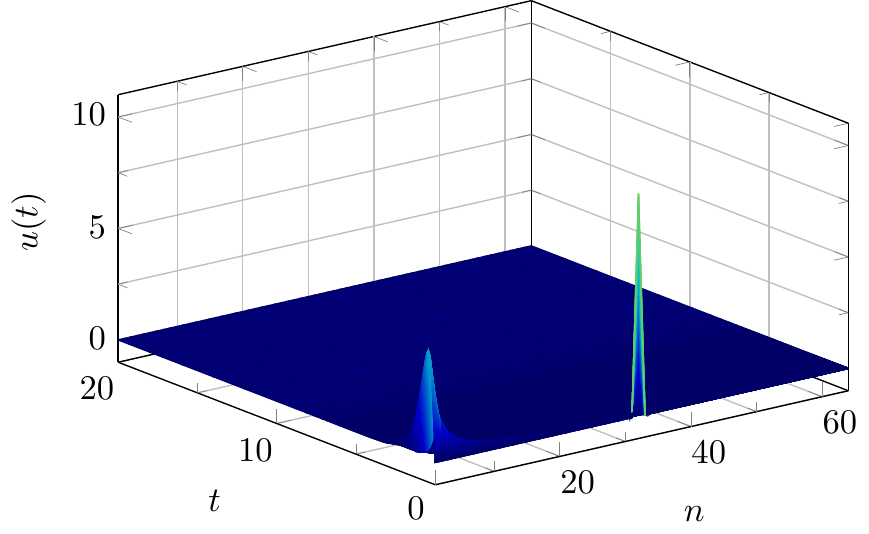}\\[24pt]
	\includegraphics{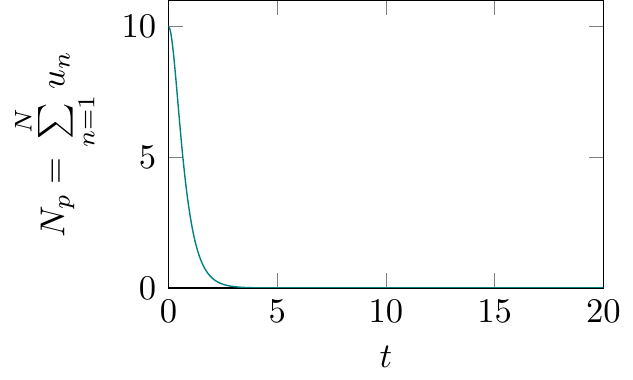} \hspace{3mm}
	\includegraphics{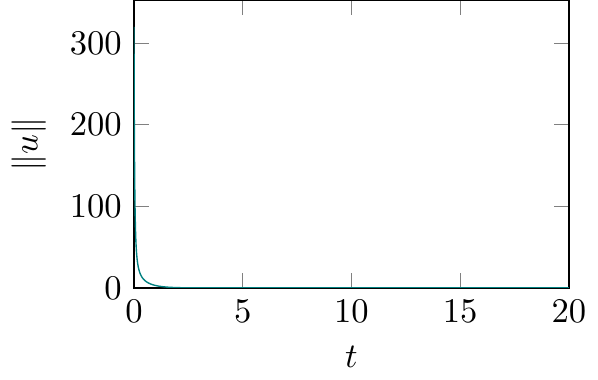}
	\caption{Evolution of the decay-fragmentation model \eqref{gf1b} with constant transport and linear fragmentation and death rates (top);
		evolution of total number of particles (bottom left); evolution of the total mass (bottom right).}\label{fig3}
\end{figure}

\subsection{Linear decay and constant fragmentation rates}\label{sec3.2}

As another illustration to Theorem~\ref{df1}, we let: $r_i = i$, $d_i = 0$, $i\ge 0$; $a_1=0$ and
$a_i=1$; and $b_{i,j} = \frac{2}{j-1}$, $i\ge 1$, $j\ge i+1$.
The dynamics of the model is shown in the top diagram of Figure~\ref{fig2}. The particles break at a constant rate
but decay at a rate faster than in the first example. The total number of particles increases initially (see the
left-bottom diagram in Fig.~\ref{fig2}), but decreases quickly due to the strong decay. The observation is
further confirmed by the right-bottom diagram in Fig.~\ref{fig2}.

\begin{figure}[t]
	\centering
	\includegraphics{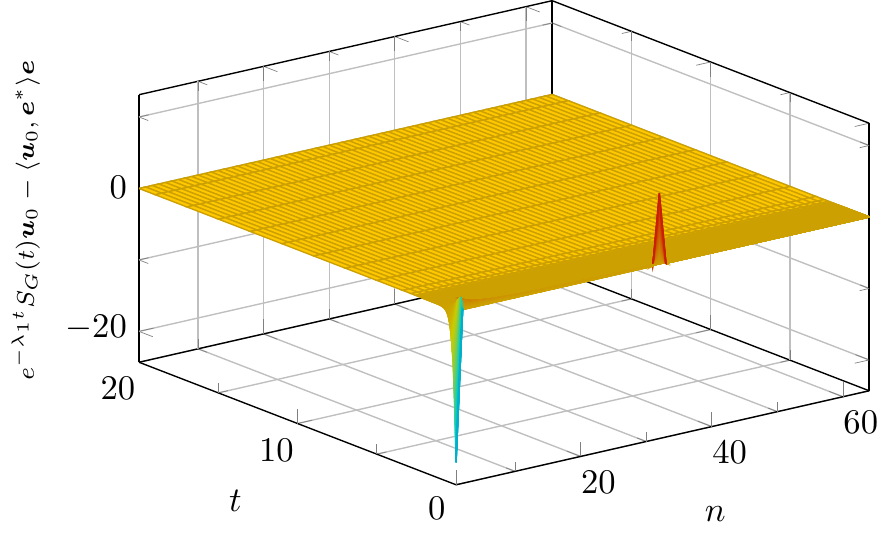}
	\includegraphics{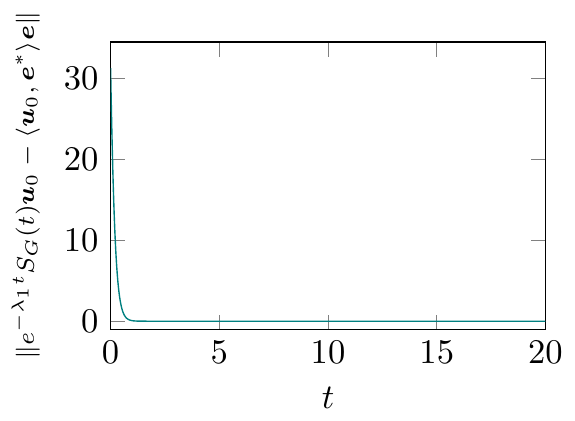}
	\caption{The asymptotic error in the decay-fragmentation model \eqref{gf1b} with constant transport and linear fragmentation and death rates.}\label{fig4}
\end{figure}

\subsection{Constant decay and linear fragmentation and death rates}\label{sec3.3}
We let $r_i = 1$, $d_i = i$, $i\ge 0$; $a_1=0$ and $a_i=i$; and $b_{i,j} = \frac{2}{j-1}$, $i\ge 1$, $j\ge i+1$.
Unlike our previous examples, the strong death rate prevents explosive growth of small clusters near $t=0$.
As time goes the solution decays steadily at a constant rate, see evolution of the total number of particles
and the total mass of the system in the left-bottom and the right-bottom diagrams of Fig.~\ref{fig3}, respectively.

Note that in this example conditions of Theorem~\ref{th3.1} are satisfied. It is easy to verify that
$\lambda_1 = -2$. Hence, we expect the numerical solution to converge to the asymptotic limit
\[
S^1_G(t)\mb u^0 = e^{\lambda_1 t}\langle\mb u_0, \mb e^*\rangle \mb e,
\]
where $\mb e^*$ and $\mb e$ are given respectively by \eqref{e*} and \eqref{eq3.10}, with $N_0 = 1$.
This is indeed the case. In complete agreement with \eqref{lt10}, after a short transition stage the gap between
$e^{-\lambda_1 t}S^1_G(t)\mb u^0$ and the projection $\langle\mb u_0, \mb e^*\rangle \mb e$
decreases exponentially as $t$ increases. The evolution of the gap
$e^{-\lambda_1 t}S^1_G(t)\mb u^0 - \langle\mb u_0, \mb e^*\rangle \mb e$ and its $X$-norm is shown
in Fig.~\ref{fig4}.

\begin{figure}[t]
	\centering
	\includegraphics{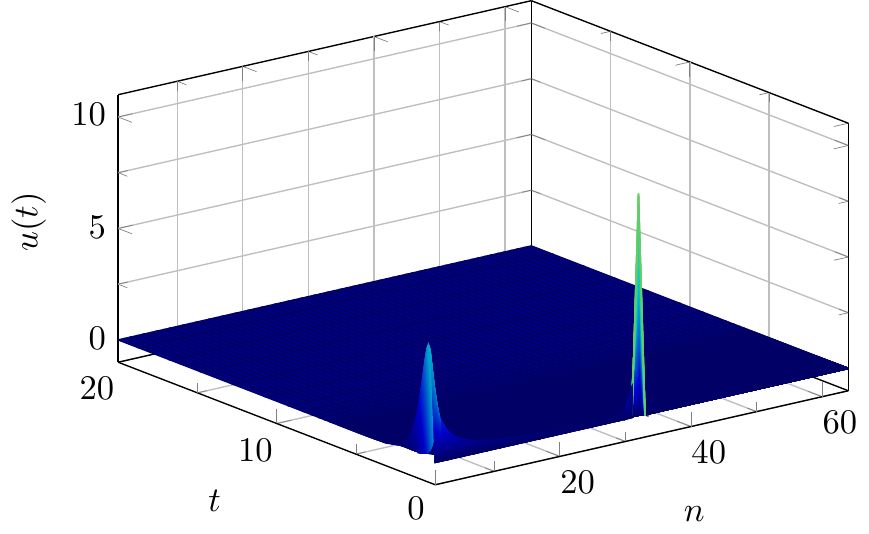}\\[24pt]
	\includegraphics{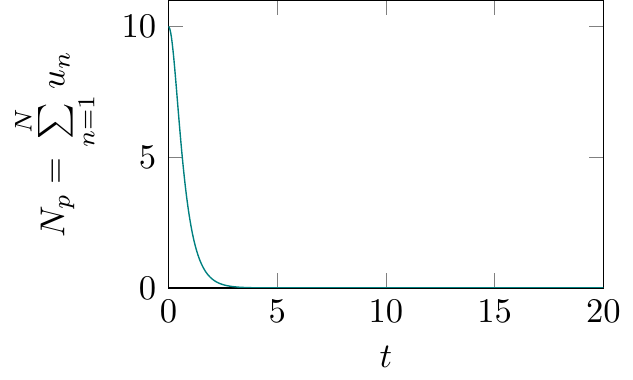} \hspace{3mm}
	\includegraphics{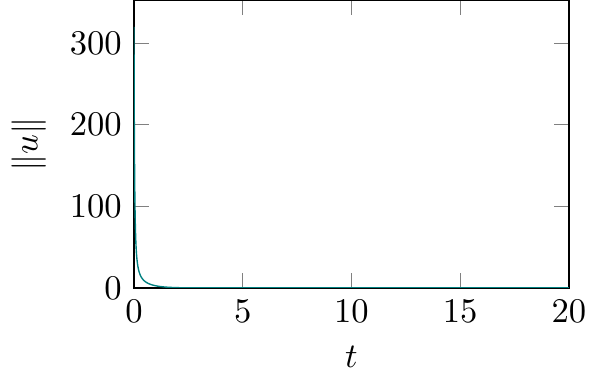}
	\caption{Evolution of the decay-fragmentation model \eqref{gf1b} with linear transport, fragmentation and
	death rates (top);
		evolution of total number of particles (bottom left); evolution of the total mass (bottom right).}\label{fig5}
\end{figure}

\begin{figure}[H]
	\centering
	\includegraphics{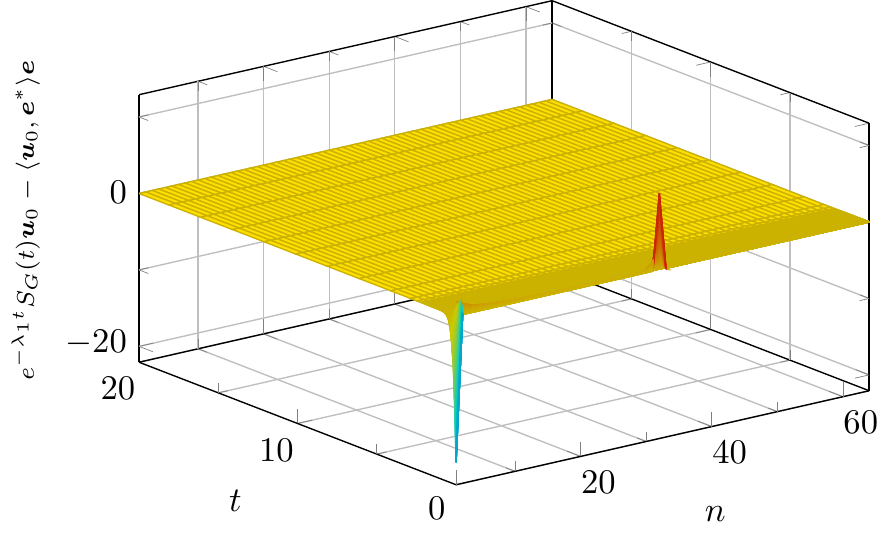}
	\includegraphics{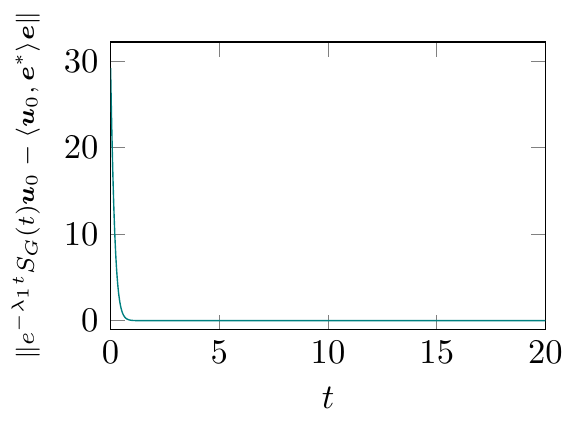}
	\caption{The asymptotic error in the decay-fragmentation model \eqref{gf1b} linear transport,
	fragmentation and death rates.}\label{fig6}
\end{figure}

\subsection{Linear decay, fragmentation and death rates}\label{sec3.4}
In our last example, we model a scenario that incorporates strong fragmentation and transport processes,
i.e. we assume $r_i = i$, $d_i = i$, $i\ge 0$; $a_1=0$ and $a_i=i$; and
$b_{i,j} = \frac{2}{j-1}$, $i\ge 1$, $j\ge i+1$. The qualitative behavior of
the numerical solution shares the dynamical features discussed in Section~\ref{sec3.3}
(see the top diagram of Fig.~\ref{fig5}). As in Sections~\ref{sec3.3}, strong death process prevents explosive
growth of small clusters near $t=0$, while strong transport yields rapid decay of the total number of particles
and the total mass of the system, see the left- and the right-bottom diagram of Fig.~\ref{fig5}, respectively.
Further, the conditions of Theorem~\ref{th3.1} are satisfied.
As a consequence, the gap $e^{-\lambda_1 t}S^1_G(t)\mb u^0 - \langle\mb u_0, \mb e^*\rangle \mb e$
demonstrates the same qualitative behaviour as in Section~\ref{sec3.3}, see Fig.~\ref{fig6}.

\bibliographystyle{apalike}
\bibliography{decayRef}

\end{document}